\documentclass{amsart}

\usepackage{amssymb}

\usepackage[matrix,arrow,curve]{xy}

\theoremstyle{plain}

\newtheorem{theorem}{Theorem}[section]

\newtheorem{corollary}[theorem]{Corollary}

\newtheorem{proposition}[theorem]{Proposition}

\newtheorem{claim}[theorem]{Claim}

\theoremstyle{definition} %actually mainly don't "mute" (in hebrew)

\newtheorem{definition}[theorem]{Definition}

\newtheorem{example}[theorem]{Example}

\newtheorem{remark}[theorem]{Remark}

\theoremstyle{remark}

\newtheorem{problem}[theorem]{Problem}

\newtheorem{equa}[theorem]{Equality}

\newtheorem{exercise}[theorem]{Exersice}

\renewcommand{\phi}{\varphi}

\newcommand{\initial}\lessdot

\def\?{?\vadjust

{\vbox to 0pt{\vskip-7pt\hbox to 1.1\hsize{\hfill\huge ?!}}}}

\newcommand{\be}{\begin{enumerate}}
\newcommand{\ee}{\end{enumerate}}

\usepackage{enumerate}
\usepackage{verbatim}
\renewcommand{\epsilon}{\varepsilon}

\usepackage{tikz}

 \def\nfork{\setbox0\hbox{$\bigcup$}%
 \setbox1=\hbox to \wd0{\hfil\vrule width 0.7pt depth 2pt height 7.5pt\hfil}%
 \wd1=0cm\relax\box1\box0}

\begin{document}

\title{Hereditary Konig Egervary Collections}

\author{Adi Jarden}
\email[Adi Jarden]{jardena@ariel.ac.il}
\address{Department of Mathematics.\\ Ariel University \\ Ariel, Israel}

\maketitle

\tableofcontents

\today

\begin{abstract}
Let $G$ be a simple graph with vertex set $V(G)$. A subset $S$ of $V(G)$ is independent if no two vertices from $S$ are adjacent. The graph $G$ is known to be a Konig-Egervary (KE in short) graph if $\alpha(G) + \mu(G)= |V(G)|$, where $\alpha(G)$ denotes the size of a maximum independent set and $\mu(G)$ is the cardinality of a maximum matching. Let $\Omega(G)$ denote the family of all maximum independent sets.

A collection $F$ of sets is an hke collection if $|\bigcup \Gamma|+|\bigcap \Gamma|=2\alpha$ holds for every subcollection $\Gamma$ of $F$.  
We characterize an hke collection and  invoke new characterizations of a KE graph. 

We prove the existence and uniqueness of a graph $G$ such that $\Omega(G)$ is a maximal hke collection. It is a bipartite graph. As a result, we solve a problem of Jarden, Levit and Mandrescu \cite{jlm}, proving that $F$ is an hke collection if and only if it is a subset of $\Omega(G)$ for some graph $G$ and $|\bigcup F|+|\bigcap F|=2\alpha(F)$. 

Finally, we show that the maximal cardinality of an hke collection $F$ with $\alpha(F)=\alpha$ and $|\bigcup F|=n$ is $2^{n-\alpha}$.
\end{abstract}

\section{Introduction}
In this paper we study hereditary KE collections: collections of sets, $F$, of a fixed cardinality $\alpha$, such that the equality
$$|\bigcup \Gamma|+|\bigcap \Gamma|=2\alpha$$
holds for every non-empty subset $\Gamma$ of $F$. 

The following theorem is a restatement of \cite[Theorem 2.6]{dam}:
\begin{theorem}\label{the main theorem of dam}
$G$ is a KE graph if and only if for some hke subcollection $\Gamma$ of $\Omega(G)$ there is a matching $M:V(G)-\bigcup \Gamma \to \bigcap \Gamma$.
\end{theorem}

In most of the sections of this paper, the main issue is the hke collections rather than graphs. In Section 2, we present characterizations for an hke collection and invoke a new characterization for a KE graph. Definition \ref{definition the typical collection} presents the typical example of a maximal hke collection. Theorem \ref{maximal iff has 2^alpha elements} characterize a maximal hke collection by its cardinality. By Corollary \ref{the typical example is the unique one}, the typical example for $\alpha$ (that is defined in Definition \ref{definition the typical collection}) is the unique maximal hke collection, where $\alpha$ is given. By Theorem \ref{bipartite}, there is a unique KE graph (actually, a bipartite graph), $G$, such that $\Omega(G)$ is a maximal hke collection. By Theorem \ref{old KE iff hke}, a collection of sets $F$ is an hke collection if and only if it is included in $\Omega(G)$ for some graph $G$ and $|\bigcup F|+|\bigcap F|=2\alpha(F)$. By Theorem \ref{corollary a(alpha,n)}, the maximal cardinality of an hke collection, $F$, with $\alpha(F)=\alpha$ and $|\bigcup F|=n$ is $2^{n-\alpha}$.

\section{Konig Egervary Collections}
 
In this section, we define a relevant collection, a KE collection and an hke collection. We charecterize hke collections and invoke new characterizations for a KE graph.

\begin{definition}
A \emph{relevant collection} is a finite collection of finite sets such that the number of elements in each set in $F$ is a constant positive integer, denoted $\alpha(F)$. When $F$ is clear from the context, we omit it, writting $\alpha$.
\end{definition}

The following definition contradicts the definition of a Konig Egervary collection in \cite{dam} (see Definition \ref{the old definition of a KE collection}).
\begin{definition}
Let $F$ be a relevant collection. $F$ is said to be a \emph{Konig Egervary collection} (KE collection in short), if the following equality
 holds: $$|\bigcup F|+|\bigcap F|=2\alpha.$$
\end{definition}

\begin{definition}\label{definition of a KE collection}
An \emph{hereditary Konig Egervary collection} (hke collection in short) is a collection of sets, $F$, such that for some positive integer, $\alpha$, the equality
$$|\bigcup \Gamma|+|\bigcap \Gamma|=2\alpha$$ holds for every non-empty sub-collection, $\Gamma$, of $F$. We call this $\alpha$, $\alpha(F)$. We may omit $F$, where it is clear from the context.
\end{definition}

\begin{proposition}\label{1}
Let $F$ be an hke collection. Clearly, for each $A \in F$, $|A|=\alpha(F)$. So $F$ is a relevant collection.
\end{proposition}

\begin{proof}
By Definition \ref{definition of a KE collection}, where we substitute $\Gamma=\{A\}$.
\end{proof}

\begin{proposition}\label{2}
Let $F$ be a relevant collection. If $|F| \leq 2$ then it is an hke collection.
\end{proposition}

\begin{proof}
It is clear when $|F|=1$. So Assume $|F|=2$,
 $F=\{A,B\}$. Take a non-empty sub-collection $\Gamma$ of $F$. Without loss of generality, $\Gamma=F$. So $$|\bigcup \Gamma|+|\bigcap \Gamma|=|A \cup B|+|A \cap B|=|A|+|B|=2\alpha(F).$$
\end{proof}

%The following proposition gives the motivation for the name `a KE collection'.
\begin{proposition}\label{the first on 2016}
Let $G$ be a KE graph. Then $\Omega(G)$ is an hke collection.
\end{proposition}

\begin{proof}
By \cite[Theorem 3.6]{jlm} and \cite[Corollary 2.8]{jlm}.
\end{proof}

%\begin{remark}
%If $F$ is an hke collection then every subcollection of $F$ is an hke collection. 
%\end{remark}

The following definition presents the typical example of a maximal hke collection for a fixed $\alpha$: by Proposition \ref{the typical example of a maximal KE collection}, it is an hke collection, by Theorem \ref{maximal iff has 2^alpha elements}, it is a maximal hke collection and by Theorem \ref{uniqueness of maximal}, it is the unique example up to isomorphism.
\begin{definition}\label{definition the typical collection}
Let $\alpha$ be a positive integer. \emph{The typical collection for $\alpha$} is the collection of subsets $S$ of $[2\alpha]$ such that $i \in S \leftrightarrow i+\alpha \notin S$ holds for every $i \in [\alpha]$.
\end{definition}
 
\begin{proposition}\label{the typical example of a maximal KE collection}
Let $\alpha$ be a positive integer. Let $F$ be the typical collection for $\alpha$. Then $F$ is an hke collection and $|F|=2^\alpha$. 
\end{proposition}

\begin{proof}
Clearly, $|F|=2^\alpha$. We prove that $F$ is an hke collection. Let $\Gamma$ be a non-empty subset of $F$. For each $i \in [\alpha]$ we have $i \in \bigcup \Gamma$ if and only if $i+\alpha \notin \bigcap \Gamma$ and $i \in \bigcap \Gamma$ if and only if $i+\alpha \notin \bigcup \Gamma$. Hence $$|\bigcup \Gamma \cap [\alpha]|+|\bigcap \Gamma \cap ([2\alpha]-[\alpha])|=\alpha,$$ and $$|\bigcap \Gamma \cap [\alpha]|+|\bigcup \Gamma \cap ([2\alpha]-[\alpha])|=\alpha.$$ Therefore $$|\bigcup \Gamma|+|\bigcap \Gamma|=2\alpha.$$ So $F$ is an hke collection.
\end{proof}

In Theorem \ref{equivalent formulations of a KE collection} we present equivalent conditions for $F$ being an hke collection.

Consider the following equality:
\begin{equa}\label{equality 1}
$$|\bigcap \Gamma_1-\bigcup \Gamma_2|=|\bigcap \Gamma_2-\bigcup \Gamma_1|.$$
\end{equa}

\begin{proposition}\label{the theorem before 23.3.16} 
Let $F$ be a relevant collection. 

 The following are equivalent:
\begin{enumerate}
\item $F$ is an hke collection. \item Equality \ref{equality 1} holds for every two non-empty disjoint sub-collections, $\Gamma_1,\Gamma_2$ of $F$, \item Equality \ref{equality 1} holds for every two non-empty disjoint sub-collections, $\Gamma_1,\Gamma_2$ of $F$ with $\Gamma_1 \cup \Gamma_2=F$.
\end{enumerate}
\end{proposition}

Before proving Proposition \ref{the theorem before 23.3.16}, we present an exercise:
\begin{exercise}\label{exercise 1}
Assume that $\{A,B,C,D\}$ is an hke collection (so in particular $\{A,B,C\}$ is an hke collection).
Prove:
\begin{enumerate}
\item $|A-B-C|=|B \cap C-A|$. A clue: $A-B-C=(A \cup B \cup C)-(B \cup C)$ and $B \cap C-A=(B \cap C)-(A \cap B \cap C)$. \item $|A \cap B-C-D|=|C \cap D-A-B|$. A clue: $A \cap B -C-D=(A-C-D)-(A-B-C-D)$. Apply Clause (1).
\end{enumerate}
\end{exercise}

We now prove Proposition \ref{the theorem before 23.3.16}.
\begin{proof}
$(1) \Rightarrow (2):$
We prove it by induction on $r=:|\Gamma_1|$. 

\emph{Case a:} $r=1$, so $\Gamma_1=\{A^*\}$ for some set $A^*$. In this case, we apply the idea of Exercise \ref{exercise 1}(1).

We should prove that $$|A^*-\bigcup \Gamma_2|=|\bigcap \Gamma_2-A^*|,$$ 

namely, $$|\bigcup \Gamma_2 \cup A^*|-|\bigcup \Gamma_2|=|\bigcap \Gamma_2|-|\bigcap \Gamma_2 \cap A^*|,$$
or equivalently,
$$|\bigcup \Gamma_2 \cup A^*|+|\bigcap \Gamma_2 \cap A^*|=|\bigcap \Gamma_2|+|\bigcup \Gamma_2|.$$
 But by Clause $(1)$, each side of this equality equals $2\alpha$.

\emph{Case a:} $r>1$. In this case, we apply the idea of Exercise \ref{exercise 1}(2). We fix $A^* \in \Gamma_1$. First we write three trivial equalities, for convenience: $$\bigcap (\Gamma_1-\{A^*\})=\{x:x \in A \text{ for every } A \in \Gamma_1 \text{ with } A \neq A^*\},$$ $$\bigcup (\Gamma_1-\{A^*\})=\{x:x \in A \text{ for some } A \in \Gamma_1 \text{ with } A \neq A^*\}$$ and $$\bigcap (\Gamma_1 \cup \{A^*\})=A^* \cap \bigcap \Gamma_1.$$ 

We now begin the computation.

$$|\bigcap \Gamma_1-\bigcup \Gamma_2|=|\bigcap (\Gamma_1-\{A^*\})-\bigcup \Gamma_2|-|\bigcap (\Gamma_1-\{A^*\})-\bigcup (\Gamma_2 \cup \{A^*\})|.$$

The right side of this equality is a subtraction of two summands. We apply the induction hypothesis on each summand: 
	
	$$|\bigcap (\Gamma_1-\{A^*\})-\bigcup \Gamma_2|=\bigcap \Gamma_2-\bigcup (\Gamma_1-\{A^*\})|$$ and
	
	$$|\bigcap (\Gamma_1-\{A^*\})-\bigcup (\Gamma_2 \cup \{A^*\})|=\bigcap (\Gamma_2 \cup \{A^*\})-\bigcup (\Gamma_1-\{A^*\})|.$$
	
By the three last equalities we get: $$|\bigcap \Gamma_1-\bigcup \Gamma_2|=|\bigcap \Gamma_2-\bigcup (\Gamma_1-\{A^*\})|-|\bigcap (\Gamma_2 \cup \{A^*\})-\bigcup (\Gamma_1-\{A^*\})|.$$

So $$|\bigcap \Gamma_1-\bigcup \Gamma_2|=|\bigcap \Gamma_2-\bigcup \Gamma_1|.$$

Equality \ref{equality 1} is proved, so Clause $(2)$ is proved.

$(2) \Rightarrow (1):$
Let $\Gamma$ be a non-empty subset of $F$.
Fix $D \in \Gamma$. Since $F$ is a relevant collection, $|D|=\alpha$ (this is the unique place where we use the assumption that $F$ is a relevant collection, but we eliminate this assumption later). Therefore it is enough to prove that
$$|\bigcup \Gamma|+|\bigcap \Gamma|=2|D|,$$
or equivalently, $$|\bigcup \Gamma-D|=|D-\bigcap \Gamma|.$$
Let $H$ be the set of ordered pairs $\langle \Gamma_1,\Gamma_2 \rangle$ of non-empty disjoint subsets of $\Gamma$ such that $\Gamma_1 \cup \Gamma_2=\Gamma$ and $D \in \Gamma_2\}$.

By Clause (2), $$\sum_{\langle \Gamma_1,\Gamma_2 \rangle \in H} |\bigcap \Gamma_1-\bigcup \Gamma_2|=\sum_{\langle \Gamma_1,\Gamma_2 \rangle \in H} |\bigcap \Gamma_2-\bigcup \Gamma_1|.$$

So it is enough to prove the following two equalities:
$$|\bigcup \Gamma-D|=\sum_{\langle \Gamma_1,\Gamma_2 \rangle \in H} |\bigcap \Gamma_1-\bigcup \Gamma_2|$$
and
$$|D-\bigcap \Gamma|=\sum_{\langle \Gamma_1,\Gamma_2 \rangle \in H} |\bigcap \Gamma_2-\bigcup \Gamma_1|.$$

Since their proofs are dual, we prove the first equality only.

$$\bigcup \Gamma-D=\bigcup_{\langle \Gamma_1,\Gamma_2 \rangle \in H} (\bigcap \Gamma_1-\bigcup \Gamma_2),$$

(on the one hand, if $x \in \bigcup \Gamma-D$ then for $\Gamma_1=\{A \in \Gamma:x \in A\}$ and $\Gamma_2=\{A \in \Gamma:x \notin A\}$ we have $x \in \bigcap \Gamma_1-\bigcup \Gamma_2$ and $\langle \Gamma_1,\Gamma_2 \rangle \in H$. On the other hand, assume that $x \in \bigcap \Gamma_1-\bigcup \Gamma_2$ for some $\langle \Gamma_1,\Gamma_2 \rangle \in H$. Then $x \in \bigcup \Gamma$ (because $x \in \bigcap \Gamma_1$ and $\emptyset \neq \Gamma_1 \subseteq \Gamma$) and $x \notin D$ (because $x \notin \bigcup \Gamma_2$ and $D \subseteq \bigcup \Gamma_2$). So $x 
\in \bigcup \Gamma-D$). Therefore

$$|\bigcup \Gamma-D|=\sum_{\langle \Gamma_1,\Gamma_2 \rangle \in H} |\bigcap \Gamma_1-\bigcup \Gamma_2|,$$
because this is a sum of cardinalities of disjoint sets (if $\langle \Gamma_1,\Gamma_2 \rangle$ and $\langle \Gamma_3,\Gamma_4 \rangle$ are two different pairs in $H$ then there is no element $x \in (\bigcap \Gamma_1-\bigcup \Gamma_2) \cap (\bigcap \Gamma_3-\bigcup \Gamma_4)$. Otherwise, take $A \in \Gamma_1-\Gamma_3$ (or vice versa). So $A \in \Gamma_4$. Hence, $x \in \bigcap \Gamma_1 \subseteq A$ and $x \notin \bigcup \Gamma_4 \supseteq A$, a contradiction).

The implication $(2) \Rightarrow (1)$ is proved.

Since Clause $(3)$ is a private case of Clause $(2)$, it remains to prove $(3) \Rightarrow (2)$. Let $\Gamma_1,\Gamma_2$ be two non-empty disjoint subsets of $F$. We should prove Equality \ref{equality 1} for these $\Gamma_1$ and $\Gamma_2$, without assuming $\Gamma_1 \cup \Gamma_2=F$. Let $H$ be the set of disjoint pairs $\langle \Gamma_1^+,\Gamma_2^+ \rangle$ of $F$ such that $\Gamma_1 \subseteq \Gamma_1^+$, $\Gamma_2 \subseteq \Gamma_2^+$ and $\Gamma_1^+ \cup \Gamma_2^+=F$.

By Clause $(3)$, 
$$\sum_{\langle \Gamma_1^+,\Gamma_2^+ \rangle \in H} |\bigcap \Gamma_1^+-\bigcup \Gamma_2^+|=\sum_{\langle \Gamma_1^+,\Gamma_2^+ \rangle \in H} |\bigcap \Gamma_2^+-\bigcup \Gamma_1^+|.$$

So it remains to prove the following two equalities:
$$|\bigcap \Gamma_1-\bigcup \Gamma_2|=\sum_{\langle \Gamma_1^+,\Gamma_2^+ \rangle \in H} |\bigcap \Gamma_1^+-\bigcup \Gamma_2^+|$$ 
and
$$|\bigcap \Gamma_2-\bigcup \Gamma_1|=\sum_{\langle \Gamma_1^+,\Gamma_2^+ \rangle \in H} |\bigcap \Gamma_2^+-\bigcup \Gamma_1^+|,$$

Since their proofs are dual, we prove the first equality only.

$$\bigcap \Gamma_1-\bigcup \Gamma_2=\bigcup_{\langle \Gamma_1^+,\Gamma_2^+ \rangle \in H} (\bigcap \Gamma_1^+-\bigcup \Gamma_2^+)$$

(on the one hand, if $x \in \bigcap \Gamma_1-\bigcup \Gamma_2$ then for $\Gamma_1=\{A \in \Gamma:x \in A\}$ and $\Gamma_2=\{A \in \Gamma:x \notin A\}$, we have $x \in \bigcap \Gamma_1^+-\bigcup \Gamma_2^+$ and the pair $\langle \Gamma_1^+,\Gamma_2^+ \rangle$ belongs to $H$. On the other hand, if $x \in \bigcap \Gamma_1^+-\bigcup \Gamma_2^+$ for some $\langle \Gamma_1^+,\Gamma_2^+ \rangle \in H$ then $x \in \bigcap \Gamma_1^+ \subseteq \bigcap \Gamma_1$ and $x \notin \bigcup \Gamma_2^+ \supseteq \bigcup \Gamma_2$. Hence, $x \in \bigcap \Gamma_1-\bigcup \Gamma_2$). Therefore

$$|\bigcap \Gamma_1-\bigcup \Gamma_2|=\sum_{\langle \Gamma_1^+,\Gamma_2^+ \rangle \in H} |\bigcap \Gamma_1^+-\bigcup \Gamma_2^+|,$$
because it is a sum of disjoint sets.
\end{proof}

The following proposition eliminates the assumption that $F$ is a relevant collection.
\begin{proposition}\label{elimination of relevant}
Clause (3) of Proposition \ref{the theorem before 23.3.16} implies that $F$ is a relevant collection. 
\end{proposition}

\begin{proof}
Define $$\alpha=\frac{|\bigcup F|+|\bigcap F|}{2}.$$ Let $D \in F$. We prove that $|D|=\alpha$. Let $P$ denote the family of partitions $\{\Gamma_1,\Gamma_2\}$ of $F$ into two non-empty subcollections.  

If $x \in \bigcup F$ then $x \in \bigcap \Gamma_1-\bigcup \Gamma_2$ for some partition $\{\Gamma_1,\Gamma_2\} \in P$ or $x \in \bigcap F$.

Let $$P_1=\{\{\Gamma_1,\Gamma_2\} \in P:D \in \Gamma_1\}$$ and $$P_2=\{\{\Gamma_1,\Gamma_2\} \in P:D \notin \Gamma_1\}.$$ 

Define $$x=\sum_{\{\Gamma_1,\Gamma_2\} \in P_1}|\bigcap \Gamma_1-\bigcup \Gamma_2|$$
and
$$y=\sum_{\{\Gamma_1,\Gamma_2\} \in P_2}|\bigcap \Gamma_1-\bigcup \Gamma_2|.$$

By Clause $(3)$ of Proposition \ref{the theorem before 23.3.16}, we have $x=y$.

It is easy to check the following three equalities:
\begin{enumerate}
\item $|\bigcup F|=x+y+|\bigcap F|=2x+|\bigcap F|$, \item
$|D|=x+|\bigcap F|$ and \item
$|\bigcup F|+|\bigcap F|=2\alpha$.
\end{enumerate}

Hence, $|D|=\alpha$. Since $D$ was an arbitrary set in $F$, $F$ is a relevant collection. 
\end{proof}

\begin{theorem}\label{equivalent formulations of a KE collection}\label{equivalent definitions of a KE collection}\label{equivalent definitions of an hke collection}
Let $F$ be a collection of sets. 

The following are equivalent:
\begin{enumerate}
\item $F$ is an hke collection. \item Equality \ref{equality 1} holds for every two non-empty disjoint sub-collections, $\Gamma_1,\Gamma_2$ of $F$, \item Equality \ref{equality 1} holds for every two non-empty disjoint sub-collections, $\Gamma_1,\Gamma_2$ of $F$ with $\Gamma_1 \cup \Gamma_2=F$.
\end{enumerate}
\end{theorem}

\begin{proof}
By Proposition \ref{the theorem before 23.3.16}, it is enough to prove that each clause implies that $F$ is a relevant collection. By Proposition \ref{1}, Clause $(1)$ implies that $F$ is a relevant collection. By Proposition \ref{elimination of relevant} Clause $(3)$ implies that $F$ is a relevant collection. But Clause $(2)$ implies Clause $(3)$.
\end{proof}

\begin{corollary}
Let $G$ be a graph. The following are equivalent:
\begin{enumerate}
\item $G$ is a KE graph. \item For some non-empty hke collection $F \subseteq \Omega(G)$, there is a matching $M:V[G]-\bigcup F \to \bigcap F$ and Equality \ref{equality 1} holds for every two non-empty disjoint sub-collections, $\Gamma_1,\Gamma_2$ of $F$. \item For some non-empty hke collection $F \subseteq \Omega(G)$, there is a matching $M:V[G]-\bigcup F \to \bigcap F$ and Equality \ref{equality 1} holds for every two non-empty disjoint sub-collections, $\Gamma_1,\Gamma_2$ of $F$ with $\Gamma_1 \cup \Gamma_2=F$.
\end{enumerate}

\end{corollary}

\begin{proof}
By Theorem \ref{equivalent formulations of a KE collection} and Theorem \ref{the main theorem of dam}.
\end{proof}

%%%%%%%%%%%%%%%%%%%%%%%%%%%%%%%%%%%%%%%%%%%%%%%
%%%
%%%%%%%%%%%%%%%%%%%%%%%%%%%%%%%%%%%%%%%%%%%%%%%
\section{Maximal hke Collections}
In this section, we study maximal hke collections. A priori it is not clear whether there is a maximal hke collection, but we prove it. Moreover, we prove the uniqueness of a maximal hke collection for a given $\alpha$ and characterize it.

The most simple characterization of a maximal hke collection is its cardinality: $|F|=2^\alpha$ (see Theorem \ref{maximal iff has 2^alpha elements}). This theorem is proved by the following argument: Let $F$ be a maximal hke collection. We fix a set $A \in F$ and present a bijection $f_A:F \to P(A)$.
\begin{definition}\label{definition of the function}
Let $F$ be an hke collection. For each $A \in F$, we define a function $f_A:F \to P(A)$ by $f_A(D)=A \cap D$.
\end{definition}

In the next propositions we present several properties of a maximal hke collection:
\begin{enumerate}
\item by Proposition \ref{the function is an injection}, $f_A$ is an injection (actually, this property holds for every hke collection), \item by Proposition \ref{maximal implies bigcap=emptyset}, $\bigcap F=\emptyset$, \item by Proposition \ref{the function is a surjection} $f_A$ is a surjection and 
\item By Theorem \ref{maximal iff has 2^alpha elements}, $|F|=2^\alpha$.
\end{enumerate}

\begin{proposition}\label{the function is an injection}
Let $F$ be an hke collection and let $A \in F$. Then the function $f_A$ is an injection of $F$ into $P(A)$. So $|F| \leq 2^\alpha$.
\end{proposition}

\begin{proof}
Let $A,B,C \in F$ with $A \cap B=A \cap C$. We have to show that $B=C$.
By symmetry, it is enough to prove that $C \subseteq B$. But $A \cap B-C=(A \cap B)-(A \cap C)= \emptyset$. So $|C-A-B|=|A \cap B-C|=0$, namely, $C-B \subseteq A$. Hence, $C-B \subseteq (A \cap C)-(A \cap B)=\emptyset$.
\end{proof}

\begin{proposition}\label{if |F|=2^alpha then F is maximal}
If $F$ is an hke collection and $|F|=2^\alpha$ then $F$ is a maximal hke collection.
\end{proposition}
 
\begin{proof}
By Proposition \ref{the function is an injection}
\end{proof}

\begin{proposition}\label{maximal implies bigcap=emptyset}
Let $F$ be a maximal hke collection. Then $\bigcap F=\emptyset$.
\end{proposition}
 
\begin{proof}
Fix $A \in F$. Define $$B=A \cup C-\bigcap F,$$ where $C$ is a set of new elements ($C \cap \bigcup F=\emptyset$) of cardinality $|\bigcap F|$. We now prove that $F'=F \cup \{B\}$ is an hke collection. First note that $|B|=|A|=\alpha$. Let $\Gamma_1$ and $\Gamma_2$ be two non-empty disjoint subsets of $F'$. We have to show that Equality \ref{equality 1} holds.

If $B \notin \Gamma_1 \cup \Gamma_2$ then it holds, because $F$ is an hke collection. Assume that $B \in \Gamma_1 \cup \Gamma_2$. Since $\Gamma_1$ and $\Gamma_2$ are disjoint, by symmetry, we may assume that $B \in \Gamma_1-\Gamma_2$. The proof is separated into four cases, according to the following two questions: 
\begin{enumerate}
\item Is $B$ the only set in $\Gamma_1$?
\item Does $A$ belong to $\Gamma_2$?
\end{enumerate} 

\emph{Case a:} $\Gamma_1=\{B\}$ and $A \in \Gamma_2$. In this case, $$\bigcap \Gamma_1-\bigcup \Gamma_2 \subseteq B-A=C \subseteq \bigcap \Gamma_1-\bigcup \Gamma_2.$$ 

So $$\bigcap \Gamma_1-\bigcup \Gamma_2=C.$$ 

Similarly, $$\bigcap \Gamma_2-\bigcup \Gamma_1=\bigcap F.$$
Since $|C|=|\bigcap F|$, Equality \ref{equality 1} holds. 

\emph{Case b:} $B \in \Gamma_1$, $1<|\Gamma_1|$ and $A \in \Gamma_2$. In this case, $$\bigcap \Gamma_1-\bigcup \Gamma_2=\emptyset=\bigcup \Gamma_2-\bigcup \Gamma_1.$$ 

\emph{Case c:} $\Gamma_1=\{B\}$ and $A \notin \Gamma_2$. In this case, $\bigcap \Gamma_1=\bigcup \Gamma_1=B$ and $$|B-\bigcup \Gamma_2|=|C|+|A-\bigcup \Gamma_2|=|C|+|\bigcap \Gamma_2-A|=|\bigcap \Gamma_2-B|,$$
(the second equality holds, because $F$ is an hke collection). So Equality \ref{equality 1} holds in this case, too.

\emph{Case d:} $B \in \Gamma_1$, $1<|\Gamma_1|$  and $A \notin \Gamma_2$. In this case, we prove the following claim:
\begin{claim}\label{17.3.2016}
Define $\Gamma_3=\Gamma_1-\{B\} \cup \{A\}$.
The following things hold:
\begin{enumerate}
\item  $\bigcap \Gamma_1-\bigcup \Gamma_2=\bigcap \Gamma_3-\bigcup \Gamma_2$, \item  
$\bigcap \Gamma_2-\bigcup \Gamma_1=\bigcap \Gamma_2-\bigcup \Gamma_3$ and
\item $|\bigcap \Gamma_3-\bigcup \Gamma_2|=|\bigcap \Gamma_2-\bigcup \Gamma_3|$.
\end{enumerate}
\end{claim}

\begin{proof}
(1) Since $\bigcap F \subseteq \bigcup \Gamma_2$, it is enough to show that if $x \in \bigcup F'- \bigcap F$ then $x \in \bigcap \Gamma_1$ if and only if $x \in \bigcap \Gamma_3$. If $x \in C$ then clearly $x \notin \bigcap \Gamma_3$, so $x$ does not belong to the right side. But since $1<|\Gamma_1|$, $x$ does not belong to $\bigcap \Gamma_1$ too. If $x \notin C$ then clearly, $x \in \bigcap \Gamma_1$ if and only if $x \in \bigcap \Gamma_3$.
 
(2) Let $x \in \bigcap \Gamma_2$. So $x \notin C$. We show that $x \in \bigcup \Gamma_1$ if and only if $x \in \bigcup \Gamma_3$. If $x \in \bigcap F$ then $x$ belongs to both $\bigcup \Gamma_1$ (because $1<|\Gamma_1|$) and $\bigcup \Gamma_3$. If $x \notin \bigcap F$ then clearly it belongs to $\bigcup \Gamma_1$ if and only if it belongs to $\bigcup \gamma_3$.

(3) Since $A \notin \Gamma_2$, the collections $\Gamma_2$ and $\Gamma_3$ are disjoint subcollections of $F$. But $F$ is an hke collection. Therefore by Theorem \ref{equivalent definitions of an hke collection}, Equality \ref{equality 1} holds for $\Gamma_2$ and $\Gamma_3$.
\end{proof}

Claim \ref{17.3.2016} implies Equality \ref{equality 1} (for $\Gamma_1$ and $\Gamma_2$).
Proposition \ref{maximal implies bigcap=emptyset} is proved.
\end{proof}

Letting an hke collection, $F$, Proposition \ref{a condition for adding a subset for KE} presents a sufficient and necessary condition for $F \cup \{D\}$ being an hke collection. 
It is a preparation for Proposition \ref{the function is a surjection}.   In \cite{broken} we present an improved version of Proposition \ref{a condition for adding a subset for KE}. 

We know by Theorem \ref{equivalent definitions of an hke collection} that $F \cup \{D\}$ is an hke collection if and only if Equality \ref{equality 1} holds for each partition of $F \cup \{D\}$. Proposition \ref{a condition for adding a subset for KE} presents a weaker condition: it is enough to consider partitions of $F \cup \{D\}$ into two subcollections, such that $A$ and $D$ belong to the same subcollection.
 
\begin{proposition}\label{a condition for adding a subset for KE}
Let $F$ be an hke collection, let $A \in F$ and let $D$ be an arbitrary set. Then $F \cup \{D\}$ is an hke collection if and only if $|D|=\alpha$ and for every partition $\{\Gamma_1,\Gamma_2\}$ of $F-\{A\}$ with $\Gamma_2 \neq \emptyset$, the following holds:
$$|A \cap D \cap \bigcap \Gamma_1-\bigcup \Gamma_2|=|\bigcap \Gamma_2-\bigcup \Gamma_1-A-D|$$
(if $\Gamma_1=\emptyset$ then we ignore $\bigcup \Gamma_1$ and $\bigcap \Gamma_1$, or more formarlly stipulate $\bigcap \Gamma_1=\bigcup F \cup \{D\}$).
\end{proposition}

Before proving Proposition \ref{a condition for adding a subset for KE}, we present an application for a relevant collection of cardinality $3$.
\begin{corollary}
Let $F'=\{A,B,D\}$ be a relevant collection of cardinality $3$. Then $F'$ is an hke collection if and only if $|A \cap D-B|=|B-A-D|$.  
\end{corollary}

\begin{proof}
$F=\{A,B\}$ is an hke collection, because its cardinality is $2$. $F-\{A\}=\{B\}$. The unique partition $\{\Gamma_1,\Gamma_2\}$ of $\{B\}$ with $\Gamma_2 \neq \emptyset$ is $\{\emptyset,\{B\}\}$. Now apply Proposition \ref{a condition for adding a subset for KE}.
\end{proof}

We now prove Proposition \ref{a condition for adding a subset for KE}.
\begin{proof}
Let us call the equality above, `Equality (*)'.

If $F \cup \{D\}$ is an hke collection then by Proposition \ref{equivalent definitions of a KE collection}($(1) \rightarrow (3)$), Equality (*) holds.

Conversely, assume that $|D|=\alpha$ and Equality (*) holds. Applying again Proposition \ref{equivalent definitions of an hke collection}($(3) \rightarrow (1)$), it remains to prove that $$ \text{(**)   } |A \cap \bigcap \Gamma_1-\bigcup \Gamma_2-D|=|D \cap \bigcap \Gamma_2-\bigcup \Gamma_1-A|,$$
for every partition $\{\Gamma_1,\Gamma_2\}$ of $F-\{A\}$ (including that case $\Gamma_2=\emptyset$).

\emph{Case A:} $\Gamma_2 \neq \emptyset$. In this case, we do not use the fact that $|D|=\alpha$. Since $F$ is an hke collection, the following equality holds:
$$ |A \cap \bigcap \Gamma_1-\bigcup \Gamma_2|=|\bigcap \Gamma_2-\bigcup \Gamma_1-A|.$$

Hence, by subtracting Equality (*) from the last equality, we obtain Equality (**). 

\emph{Case B:} $\Gamma_2=\emptyset$. In this case, we have to prove $$|\bigcap F-D|=|D-\bigcup F|.$$
 We use the assumption $|D|=\alpha$ and Case A. 

\begin{claim}\label{equality ***}
$$|\bigcap \Gamma_1-\bigcup \Gamma_2-D|=|D \cap \bigcap \Gamma_2-\bigcup \Gamma_1|$$
holds, for each partition $\{\Gamma_1,\Gamma_2\}$ of $F$ into two non-empty subcollections.
\end{claim}

\begin{proof}
By Equality (**) in case A and by Equality (*).
\end{proof}

Consider the following partition of $D$:
$$\mathbb{D}=\{D \cap \bigcap \Gamma_2-\bigcup \Gamma_1:\{\Gamma_1,\Gamma_2\} \text{ is a partition of } F\}.$$

For every set $D' \in \mathbb{D}$, we define a `pseudo dual set', $(D')^d$, that is a subset of $A$ as follows:

\begin{displaymath}
(D \cap \bigcap \Gamma_2-\bigcup \Gamma_1)^d=\Big\{
\begin{matrix}
 D \cap \bigcap \Gamma_2-\bigcup \Gamma_1, && \text{ if } A \in \Gamma_2
\\ \bigcap \Gamma_1-\bigcup \Gamma_2-D, &&
\text{ if } A \notin \Gamma_2
\end{matrix}
\end{displaymath}

By Claim \ref{equality ***}, $|(D')^d|=|D'|$ holds for each $D' \in \mathbb{D}$, except maybe two sets: $D \cap \bigcap F$ and $D-\bigcup F$. But $(D \cap \bigcap F)^d=D \cap \bigcap F$. So the equality $|(D')^d|=|D'|$ holds for $D \cap \bigcap F$ as well. Our goal is to prove this equality for $D-\bigcup F$: to prove that $|(D-\bigcup F)^d|=|D-\bigcup F|$ (namely, $|\bigcap F-D|=|D-\bigcup F|$). 

Define $$\mathbb{A}=\{(D')^d:D' \in \mathbb{D}\}.$$

Since $|D|=\alpha=|A|$, if $\mathbb{A}$ is a partition of $A$ then $$\sum_{D' \in \mathbb{D}}|D'|=|D|=|A|=\sum_{D' \in \mathbb{D}}|(D')^d|$$
and so the equality $|(D')^d|=|D'|$ holds for $D'=D-\bigcup F$ as needed.

By the following claim, $\mathbb{A}$ is a partition of $A$:
\begin{claim}
$\mathbb{A}$ satisfies the following things: 
\begin{enumerate}
\item  $\bigcup \mathbb{A} \subseteq A$, \item $A \subseteq \bigcup \mathbb{A}$ and \item  the pseudo dual sets of each two different sets in $\mathbb{D}$ are disjoint.
\end{enumerate}
\end{claim}

\begin{proof}
(1) Clear.

(2) Let $x \in A$. Define $\Gamma_2=\{B \in F:x \in B\}$ and $\Gamma_1=F-\Gamma_2$. So $x \in \bigcap \Gamma_2-\bigcup \Gamma_1$ and $A \in \Gamma_2$. If $x \in D$ then $x \in D \cap \bigcap \Gamma_2-\bigcup \Gamma_1 \in \mathbb{A}$. So $x \in \bigcup \mathbb{A}$. 

If $x \notin D$ then $x \in \bigcap \Gamma_2-\bigcup \Gamma_1-D \in \mathbb{A}$, because $\bigcap \Gamma_2-\bigcup \Gamma_1-D$ is the pseudo dual of $D \cap \bigcap \Gamma_1-\bigcup \Gamma_2$.  So $x \in \bigcup \mathbb{A}$.

(3) Take two different partitions $\{\Gamma_1^a,\Gamma_2^a\}$ and $\{\Gamma_1^b,\Gamma_2^b\}$ of $F$. 

If $A \in \Gamma_2^a-\Gamma_2^b$ then the pseudo dual sets are $D \cap \bigcap \Gamma_2^a-\bigcup \Gamma_1^a$ and $\bigcap \Gamma_1^b-\bigcup \Gamma_2^b-D$. The first is included in $D$ while the second is disjoint to $D$.

If $A \in \Gamma_2^a \cap \Gamma_2^b$ then the pseudo dual sets are $D \cap \bigcap \Gamma_2^a-\bigcup \Gamma_1^a$ and $D \cap \bigcap \Gamma_2^b-\bigcup \Gamma_1^b$. Since $\Gamma_2^a \neq \Gamma_2^b$, there is a set $B \in \Gamma_2^a-\Gamma_2^b$ or vice versa. So one pseudo dual set is included in $B$, while the other is disjoint to $B$.

If $A \notin \Gamma_2^a \cup \Gamma_2^b$ then the pseudo dual sets are $\bigcap \Gamma_1^a-\bigcup \Gamma_2^a-D$ and $\bigcap \Gamma_1^b-\bigcup \Gamma_2^b-D$. So again there is a set $B \in \Gamma_1^a-\Gamma_1^b$ or vice versa. Therefore one pseudo dual set is included in $B$, while the other is disjoint to $B$.
\end{proof}
Proposition \ref{a condition for adding a subset for KE} is proved.
\end{proof}

By the following proposition, if $F$ is a maximal hke collection, then for each $A \in F$, $f_A$ is
onto the power set of $A$. Moreover, it specifies several properties of a set $D$ that can be added to $F$, if $f_A$ is not a surjection.
\begin{proposition}\label{the function is a surjection}
Let $F$ be an hke collection. Let $E \subseteq A \in F$. Then we can find a set $D$ such that the following holds:
\begin{enumerate}
\item $A \cap D=E$, 
\item $F \cup \{D\}$ is an hke collection and \item $|D-\bigcup F|=|\bigcap F-E|$.
\end{enumerate}
 In particular, if $\bigcap F=\emptyset$ then $D \subseteq \bigcup F$.

%I have to delete the following bad version: complete...
%Let $F$ be an hke collection. Let $A \in F$. If for some subset $E$ of $A$, there is no $D \in F$ with $D \cap A=E$ then we can find $D \notin F$ such that $D \subseteq \bigcup F \cup C$ and $F \cup \{D\}$ is an hke collection, where $C$ is a set of new elements of cardinality $|\bigcap F-E|$. In particular, if $\bigcap F=\emptyset$ (and $f_A$ is not onto $P(A)$) then $D \subseteq \bigcup F$.
\end{proposition}

It is recommended to assume $\bigcap F=\emptyset$ at the first reading of the following proof. 
\begin{proof} 
Let us give a rough description of the proof: The needed $D$ should be a set of cardinality $\alpha$ that includes $E$ with $A \cap D=E$. So $D$ will be the union of $E$ with $\alpha-|E|$ elements that are not in $A$. For each non-empty subset $\Gamma$  of $F-\{A\}$ the `dual set' has the same number of elements. But we subtract $e_\Gamma$ elements from the dual set, so that at the end, we will subtract $|E|$ elements. But if $\Gamma=\emptyset$ then there is no dual set and we have to choose elements outside of $\bigcup F$.

For every subset $\Gamma$ of $F$ such that $A \in \Gamma$, we define $e_\Gamma$ as follows: Define $\Gamma_1=\Gamma$ and $\Gamma_2=F-\Gamma_1$. Define $$e_\Gamma=:|E \cap \bigcap \Gamma_1-\bigcup \Gamma_2|.$$ 

Since $E \subseteq A$, the following equality holds: $$(1) \sum_{A \in \Gamma \subseteq F}e_\Gamma=|E|.$$

By Proposition \ref{equivalent definitions of a KE collection}, if $\Gamma_2$ is not empty (or equivalently, $\Gamma \neq F$) then
$$|\bigcap \Gamma_2-\bigcup \Gamma_1|=|\bigcap \Gamma_1-\bigcup \Gamma_2| \geq e_\Gamma.$$

We choose a set, $E_\Gamma$, of $e_\Gamma$ elements in $\bigcap \Gamma_2-\bigcup \Gamma_1$ for each such $\Gamma$ ($E_\Gamma$ is the set of elements in $\bigcap \Gamma_2-\bigcup \Gamma_1$ that are not going to be in the needed set, $D$). 
 
Let $C$ be a set of cardinality $|\bigcap F|-e_F=|\bigcap F-E|$ that is disjoint to $\bigcup F$. Define $$D=:E \cup [\bigcup F-A-\bigcup_{A \in \Gamma \subsetneq F} E_\Gamma] \cup C.$$ 

Clearly $D \cap A=E$.

%I have to delete this bad version: complete... Clearly, $D \notin F$, because $D \cap A=E$, while by assumption for each $B \in F$, we have $A \cap B \neq E$. 

It remains to show that $F \cup \{D\}$ is an hke collection. In order to apply Proposition \ref{a condition for adding a subset for KE}, we firstly have to prove that $|D|=\alpha$.

$$|D|=|E|+|\bigcup F-A|-\Sigma_{A \in \Gamma \subsetneq F}e_\Gamma+|C|.$$

By Equality (1) and the equality $|C|=|\bigcap F|-e_F$, we get 

$$|D|=|E|+|\bigcup F-A|-\Sigma_{A \in \Gamma \subset F}e_\Gamma+|\bigcap F|=|\bigcup F|+|\bigcap F|-|A|=2\alpha-\alpha=\alpha.$$

The equality $|D|=\alpha$ is proved. It remains to show that the second condition of Proposition \ref{a condition for adding a subset for KE} is satisfied.

Let $\{\Gamma_1,\Gamma_2\}$ be a partition of $F$ such that $A \in \Gamma_1$ and $\Gamma_2 \neq \emptyset$. 
  
We have to show that

$$|D \cap \bigcap \Gamma_1-\bigcup \Gamma_2|=|\bigcap \Gamma_2-\bigcup \Gamma_1-D|.$$

Define $\Gamma=\Gamma_1$. Since $D \cap A=E$ and $A \in \Gamma_1$, 
 $$|D \cap \bigcap \Gamma_1-\bigcup \Gamma_2|=|E \cap \bigcap \Gamma_1-\bigcup \Gamma_2|=e_\Gamma.$$

It remains to show that $$|\bigcap \Gamma_2-\bigcup \Gamma_1-D|=e_\Gamma.$$ 

By the definition of $D$ and the fact that the sets $E$ and $C$ are disjoint to the set $\bigcap \Gamma_2-\bigcup \Gamma_1$, we have

$$\bigcap \Gamma_2-\bigcup \Gamma_1-D=\bigcap \Gamma_2-\bigcup \Gamma_1-(\bigcup F-A-\bigcup_{A \in \Gamma' \subsetneq F}E_{\Gamma'}).$$

Since $\bigcap \Gamma_2-\bigcup \Gamma_1 \subseteq \bigcup F-A$, the right side of the last equality equals

$$\bigcap \Gamma_2-\bigcup \Gamma_1-(\bigcap \Gamma_2-\bigcup \Gamma_1-\bigcup_{A \in \Gamma' \subsetneq F}E_{\Gamma'}),$$

namely, $$(\bigcap \Gamma_2-\bigcup \Gamma_1) \cap \bigcup_{A \in \Gamma' \subsetneq F}E_{\Gamma'}=E_\Gamma$$

(the last equality holds, because for each $\Gamma' \neq \Gamma$, the sets $E_{\Gamma'}$ and $\bigcap \Gamma_2-\bigcup \Gamma_1$ are disjoint). 

Hence, 
$$|\bigcap \Gamma_2-\bigcup \Gamma_1-D|=|E_\Gamma|=e_\Gamma.$$

\begin{comment}
%I postpone deleting the following version, until I check if the proposition presenting a condition for $F \cup \{D\}$ being KE is correct.
Let $l \in L-\{\bar{1}\}$. we have to show that

$$|D \cap \bigcap_{i=1}^k A_i^{l(i)}|=|\bigcap_{i=1}^k A_i^{1-l(i)}-D|.$$

\emph{Case a:} $l(1)=1$. In this case, on the one hand,
 $$|D \cap \bigcap_{i=1}^k A_i^{l(i)}|=|D \cap A \cap \bigcap_{i=2}^k A_i^{l(i)}|=|E \cap \bigcap_{i=2}^k A_i^{l(i)}|=e_l$$
and on the other hand
$$|\bigcap_{i=1}^k A_i^{1-l(i)}-D|=|\bigcap_{i=2}^k A_i^{1-l(i)}-D-A|=|E_l|=e_l.$$

%Do I have to delete Case b? It depends on the trues value of the preceding proposition.

\emph{Case b:} $l(1)=0$. In this case,
on the one hand,
 $$|D \cap \bigcap_{i=1}^k A_i^{l(i)}|=|D \cap \bigcap_{i=2}^k A_i^{l(i)}-A|=|\bigcap_{i=2}^k A_i^{l(i)}-A-E_{1-l}|=|\bigcap_{i=2}^k A_i^{l(i)}-A|-e_{1-l}$$
and on the other hand
$$|\bigcap_{i=1}^k A_i^{1-l(i)}-D|=|A \cap \bigcap_{i=2}^k A_i^{1-l(i)}-D|=|A \cap \bigcap_{i=2}^k A_i^{1-l(i)}-E|=|A \cap \bigcap_{i=2}^k A_i^{1-l(i)}|-e_{1-l}.$$
But since $F$ is an hke collection, $$|A \cap \bigcap_{i=2}^k A_i^{1-l(i)}|=|\bigcap_{i=2}^k A_i^{1-l(i)}-A|.$$ 
\end{comment}
\end{proof}

\begin{remark}
We now can present a new proof of Proposition \ref{maximal implies bigcap=emptyset}. Assume that $\bigcap F \neq \emptyset$. Define $E=\emptyset$. Let $A \in F$. $E \subseteq A$ is not in the image of $f_A$, because $A \cap D =\bigcap F \neq \emptyset$ for each $D \in F$. So $f_A$ is not a surjection and by Proposition \ref{the function is a surjection}, $F$ is not a maximal hke collection. 
\end{remark}

\begin{theorem}\label{maximal iff has 2^alpha elements}
Let $F$ be an hke collection. Then $F$ is a maximal hke collection if and only if $|F|=2^{\alpha}$.
\end{theorem}

\begin{proof}
If $|F|=2^\alpha$ then by Proposition \ref{if |F|=2^alpha then F is maximal}, $F$ is a maximal hke collection.

Conversely, assume that $F$ is a maximal hke collection. Fix $A \in F$. By Propositions \ref{the function is an injection} and \ref{the function is a surjection}, $f_A:F \to P(A)$ is a bijection. So $|F|=2^{\alpha}$.
\end{proof}

\section{The Uniqueness of a Maximal hke Collection}
In this section (Theorem \ref{another characterization of a maximal KE collection}), we present two similar characterizations of a maximal hke collection and use it to prove the uniqueness of a maximal hke collection for a fixed $\alpha$ (Theorem \ref{uniqueness of maximal}). Moreover, this characterization yields a connection between maximal hke \emph{collections} and minimal KE \emph{graphs} (Theorem \ref{bipartite}).

The following definition is needed in order to state Theorem \ref{another characterization of a maximal KE collection}.
\begin{definition}
Let $F$ be an hke collection. We define a relation $\approx$ on $\bigcup F$ as follows: $x \approx y$ if and only if $x=y$ or $A-D=\{x\}$ and $D-A=\{y\}$ for some $A,D \in F$. If $F$ is not clear from the context, we will write $\approx_F$. If $F$ is a maximal hke collection, $x \neq y$ and $x \approx y$ then $y$ is said to be \emph{the dual} of $x$ (this name is justified in the proof of Theorem \ref{another characterization of a maximal KE collection}). If $F$ is not clear from the context, we will write `the $F$-dual'.
\end{definition}

Clauses (2) and (3) of Theorem \ref{another characterization of a maximal KE collection} are similar. In Clause (2), we require the existence of an equivalence relation satisfying several specific properties. In Clause (3), we require that $\approx$ will satisfy these properties.
\begin{theorem}\label{another characterization of a maximal KE collection}
Let $F$ be a relevant collection with $\alpha(F)=\alpha$. The following things are equivalent:
\begin{enumerate}
%Let $A_\alpha$ be as in the above proposition.
\item $F$ is a maximal hke collection. \item $|\bigcup F|=2\alpha$ and there is an equivalence relation on $\bigcup F$ with $\alpha$ equivalence classes, each has two elements such that $F$ equals the collection of subsets $B$ of $\bigcup F$ such that $|B|=\alpha$ and $B$ intersects each equivalence class. \item $|\bigcup F|=2\alpha$, $\approx$ has $\alpha$ equivalence classes, each equivalence class has two elements and $F$ equals the collection of subsets $B$ of $\bigcup F$ such that $|B|=\alpha$ and $B$ intersects each $\approx$-equivalence class.
\end{enumerate}
\end{theorem}

\begin{proof}
Clause (3) implies Clause (2) trivially. Assume that Clause (2) holds. By renaming, without loss of generality, $\bigcup F=[2\alpha]$ and the equivalence classes are the pairs of the form $\{m,m+\alpha\}$ for $1 \leq m \leq \alpha$. Then $F$ is the typical collection for $\alpha$. So by Proposition \ref{the typical example of a maximal KE collection}, $F$ is an hke collection and $|F|=2^\alpha$. So by Proposition \ref{if |F|=2^alpha then F is maximal}, Clause (1) holds.

It remains to prove that Clause (1) implies Clause (3). Assume that $F$ is a maximal hke collection. In Claims \ref{A and D of x}-\ref{approx}, we prove that $\approx$ is an equivalence relation and study its properties.
\begin{claim}\label{A and D of x}
For every $x \in \bigcup F$, there are two sets $A$ and $D$ in $F$ such that $A-D=\{x\}$.
\end{claim}

\begin{proof}
Take $A \in F$ such that $x \in A$. By Proposition \ref{the function is a surjection}, $f_A$ is a surjection. So there is a set $D \in F$ such that $A \cap D=A-\{x\}$. Therefore $A-D=\{x\}$.
\end{proof}

By Claim \ref{approx is transitive and more} (below), the relation $\approx$ is an equivalence relation. By the following claim, each $\approx$-equivalence class has two elements at least.
\begin{claim}\label{at least 2}
For every element $x \in \bigcup F$, there is an element $y \in \bigcup F$ such that $x \approx y$ and $x \neq y$.
\end{claim}

\begin{proof}
By Claim \ref{A and D of x}, there are two sets, $A$ and $D$ in $F$, such that $A-D=\{x\}$. By Theorem \ref{equivalent formulations of a KE collection}, $|D-A|=|A-D|=1$. Let $y$ be the unique element in $D-A$. Then $x \approx y$. 
\end{proof}

\begin{claim}\label{for some implies for all}
Let $x,y$ be two different elements in $\bigcup F$. Then $x \approx y$ if and only if $x \in B \leftrightarrow y \notin B$ holds for every $B \in F$.
\end{claim}

\begin{proof}
First assume that $x \in B \leftrightarrow y \notin B$ holds for every $B \in F$. By Claim \ref{A and D of x}, there are two sets $A$ and $D$ in $F$, such that $A-D=\{x\}$. So $|D-A|=|A-D|=1$. But $y \in A-D$. Therefore $D-A=\{y\}$ and $x \approx y$. 

Conversely, assume that $x \approx y$. So there are $A$ and $D$ in $F$ such that $A-D=\{x\}$ and $D-A=\{y\}$. Let $B \in F$. We show that $x \in B$ if and only if $y \notin B$ using the following three equalities:
\begin{enumerate}
\item $A \cap B-D=(A-D) \cap B= \{x\} \cap B$, 
\item $|A \cap B-D|=|D-A-B|$ (By Theorem \ref{equivalent formulations of a KE collection}) and \item $D-A-B=\{y\}-B$ .
\end{enumerate}

$x \in B$ if and only if $\{x\} \cap B \neq \emptyset$ if and only if $|A \cap B-D| \neq 0$ if and only if $|D-A-B| \neq 0$ if and only if $\{y\}-B \neq \emptyset$, namely $y \notin B$.

\end{proof}
 
By the following claim, the relation $\approx$ is transitive, so it is an equivalence relation. Moreover, each equivalence class has two elements at most. 
\begin{claim}\label{approx is transitive and more}
If $x \neq y$, $x \approx y$, $y \neq z$ and $y \approx z$ then $x=z$.
\end{claim}

[Here is another formulation of Claim \ref{approx is transitive and more}: if $x$ is the dual of $y$ and $z$ is the dual of $y$ then $x=z$.] 

\begin{proof}
By Claim \ref{for some implies for all}, for every $B \in F$, we have $$x \in B \Leftrightarrow y \notin B \Leftrightarrow z \in B.$$

By Claim \ref{A and D of x}, there are two sets, $A,D \in F$ such that $A-D=\{x\}$. Since $x \in A-D$, we have $z \in A-D$, as well. So $z \in \{x\}$, namely, $z=x$.
\end{proof}

\begin{claim}\label{approx}
The relation $\approx$ is an equivalence relation. It has $\alpha$ equivalence classes, each has two elements exactly.
\end{claim}

\begin{proof}
By its definition, the relation $\approx$ is reflexive and symmetric. By Claim \ref{approx is transitive and more}, $\approx$ is transitive. So it is an equivalence relation. By Claim \ref{at least 2}, each equivalence class of $\approx$ has two elements at least. By Claim \ref{approx is transitive and more} each equivalence class of $\approx$ has two elements at most. So each equivalence class of $\approx$ has exactly two elements. 

We now show that there are exactly $\alpha$ $\approx$-equivalence classes. Fix $A \in F$. Let $f$ be the function of the set of $\approx$-equivalence classes to $A$, such that $f(\{x,y\})$ is the unique element in $\{x,y\} \cap A$. The number of $\approx$-equivalence classes is $\alpha$, because $f$ is a bijection: Let $\{x,y\}$ be an $\approx$-equivalence class. By Claim \ref{for some implies for all}, $x \notin A$ if and only if $y \in A$. So $f$ is well-defined. Since each $x \in A$ is in some $\approx$-equivalence class, $f$ is surjective. Since an element in $A$ cannot be in two different $\approx$-equivalence classes at the same time, $f$ is injective.  
\end{proof} 

Let $F'$ be the collection of subsets $B$ of $\bigcup F$ such that $|B|=\alpha$ and $B$ intersects each equivalence class of $\approx$. By Claim \ref{approx}, $F'$ is the collection of subsets $B$ of $\bigcup F$ such that $|B \cap \{x,y\}|=1$ for each equivalence class $\{x,y\}$ of $\approx$. So $|F'|=2^\alpha$. 

It remains to prove that $F'=F$. Since $F$ is a maximal hke collection, by Theorem \ref{maximal iff has 2^alpha elements}, $|F|=2^\alpha$ too. So it is enough to prove that $F \subseteq F'$. Let $B \in F$. So $|B|=\alpha$. Let $\{x,y\}$ be an $\approx$-equivalence relation. Since $x \approx y$, By Claim \ref{for some implies for all}, $x \in B$ if and only $y \notin B$. So $\{x,y\} \cap B \neq \emptyset$. Therefore $B \in F'$. Hence, $F'=F$. Clause (3) is proved. So Theorem \ref{another characterization of a maximal KE collection} is proved.
\end{proof}
 
\begin{definition}
Let $F_1$ and $F_2$ be two collections of sets. We say that $F_1$ and $F_2$ are \emph{isomorphic} when: there is a bijection $g:\bigcup F_1 \to \bigcup F_2$ such that $g[S] \in F_2$ for every $S \in F_1$ and $g^{-1}[S] \in F_1$ for every $S \in F_2$.
\end{definition}

By the following theorem, up to isomorphism, there is a unique maximal hke collection with $\alpha$ fixed: 
\begin{theorem}\label{uniqueness of maximal}
If $F_1$ and $F_2$ are two maximal hke collections with $\alpha(F_1)=\alpha(F_2)$ then $F_1 \cong F_2$. 
\end{theorem}

\begin{proof}
Define $\alpha=\alpha(F_1)=\alpha(F_2)$. Fix $A_1 \in F_1$ and $A_2 \in F_2$. So $|A_1|=|A_2|=\alpha$. Let $f:A_1 \to A_2$ be a bijection. Extend $f$ to a function $g:\bigcup F_1 \to \bigcup F_2$ such that if $y$ is the dual of $x$ in $F_1$ then $g(y)$ is the dual of $g(x)$ in $F_2$ (it holds vacuously for $x,y \in A_1$, because in this case, $y$ is not the dual of $x$). So for every two elements $x,y \in \bigcup F_1$, $y$ is the dual of $x$ in $F_1$ if and only if $g(y)$ is the dual of $g(x)$ in $F_2$. By Claim \ref{approx}, $g$ is bijective.

By the symmetry between $F_1$ and $F_2$, in order to prove that $g$ is an isomorphism, it is enough to show that $g[S]$ belongs to $F_2$ for each $S \in F_1$. By Theorem \ref{another characterization of a maximal KE collection}((1) $\rightarrow$ (3)), for every subset $B$ of $\bigcup F_2$ the following holds: $B \in F_2$ if and only if $\alpha \leq |B|$ and each two different elements in $B$ are not $\approx_{F_2}$-equivalent. 

Since $g$ is injective, $\alpha=|S| \leq g[S]$. Let $g(x),g(y)$ be two different elements in $g[S]$. If $g(x) \approx_{F_2} g(y)$ then $g(y)$ is the dual of $g(x)$ and so $y$ is the dual of $x$. But $y$ cannot be the dual of $x$, because they are both in $S$.
\end{proof}

\begin{corollary}\label{the typical example is the unique one}
Let $F$ be a relavant collection with $\alpha(F)=\alpha$. Then $F$ is a maximal hke collection if and only if $F$ is isomorphic to the typical collection for $\alpha$. 
\end{corollary}

\begin{proof}
By Theorem \ref{uniqueness of maximal}, we have only to prove that if $F$ is the typical collection for $\alpha$ then $F$ is a maximal hke collection. Assume that $F$ is the typical collection for $\alpha$. By Proposition \ref{the typical example of a maximal KE collection}, $F$ is an hke collection and $|F|=2^\alpha$. So by Theorem \ref{maximal iff has 2^alpha elements}, $F$ is a maximal hke collection.
\end{proof}

The following corollary will be used in the proof of Theorem \ref{bipartite}.
\begin{corollary}\label{maximal independent is realy}
Let $F$ be a maximal hke collection. Let $x$ and $y$ be two different elements in $\bigcup F$. Then the following conditions are equivalent:
\begin{enumerate}
\item $x \approx y$, \item $\{x,y\} \cap S \neq \emptyset$ for each $S \in F$ and \item $\{x,y\} \nsubseteq S$ for each $S \in F$.
\end{enumerate}
\end{corollary}

\begin{proof}
By Corollary \ref{the typical example is the unique one}, without loss of generality, $F$ is the typical collection for $\alpha$ (in particular, $x$ and $y$ are numbers in $[2\alpha]$). So each condition ((1), (2) and (3)) is equivalent to the condition $|x-y|=\alpha$.
\end{proof}

%%%%%%%%%%%%%%%%%%%%%%%%%%%%%%%%
%%%
%%%%%%%%%%%%%%%%%%%%%%%%%%%%%%%%

\section{From a Collection to a Graph}
 
\begin{definition}
Let $F$ be a relevant collection. The \emph{graph of $F$}, $G(F)$, is the graph $(V(G),E(G))$, where $V(G)=:\bigcup F$ and $E(G)=:\{vu:v,u \in G$ and there is no $A \in F$ such that $\{u,v\} \subseteq A$\}.
\end{definition}

$G(F)$ has the maximal set of edges such that each set in $F$ is independent.

\begin{proposition}\label{alpha(F) leq alpha(G)}
Let $F$ be a relevant collection. Then $\alpha(F) \leq \alpha(G(F))$.
\end{proposition}

\begin{proof}
Every set in $F$ is independent in $G(F)$. Therefore $\alpha(F) \leq \alpha(G(F))$.
\end{proof}

It is easy to find a relevant collection, $F$, such that $\alpha(F)<\alpha(G(F))$.
\begin{example}
$$F:=\{\{1,2\},\{1,3\},\{2,3\}\}$$
Here, $\alpha(F)=2$, while $\alpha(G(F))=3$.
\end{example}

\begin{example}
$$F:=\{\{0,1,2,3,4,8\},\{0,3,4,5,6,7\},\{0,1,2,5,6,9\}\}$$
Here, $\alpha(F)=6$, while $\alpha(G(F))=7$.
\end{example} 

Recall:
\begin{definition}
\emph{A well-covered graph} is a graph in which every independent set can be extended to a maximum independent set.
\end{definition}

\begin{proposition}
For every well-covered graph $G$ with $V(G)=corona(G)$, the graph of $\Omega(G)$ is $G$.
\end{proposition}

\begin{proof}
Let $G$ be a well-covered graph with $V(G)=corona(G)$ and let $G'$ be the graph of $\Omega(G)$. We should prove that $G=G'$. They have the same set of vertices:
$$V(G')=\bigcup \Omega(G)=corona(G)=V(G).$$

They have the same set of edges as well:
First assume that $(uv) \in E(G)$. $\{u,v\} \nsubseteq X$, for every $X \in \Omega(G)$. So $(uv) \in E(G')$. Conversely, assume that $(uv) \notin E(G)$. So $\{u,v\}$ is an independent set in $G$. Since $G$ is well-covered, we can find a set $X \in \Omega(G)$ such that $\{u,v\} \subseteq X$. Therefore $(uv) \notin E(G')$.
\end{proof}

The next example exemplifies the following facts:
\begin{enumerate}
\item there is an hke collection $F$ such that $\Omega(G) \neq F$ for every graph $G$ and \item we can find two different hke collections, $F_1$ and $F_2$ such that $G(F_1)=G(F_2)$.
\end{enumerate}

\begin{example} 
$$F:=\{\{1,3,5\},\{1,4,6\},\{2,3,5\},\{2,4,5\},\{2,4,6\}\}.$$ $F$ is an hke collection and $G(F)$ is the bipartite graph $G=(V,E)$ with $V=\{1,2,3,4,5,6\}$ and $E=\{12,34,36,56\}$. But $$\Omega(G)=\{\{1,3,5\},\{1,4,6\},\{2,3,5\},\{2,4,5\},\{2,4,6\},\{1,4,5\}\}=F \cup \{1,4,5\} \neq F.$$

Actually, there is no graph $G$ with $\Omega(G)=F$. For let $G$ be a graph with $F \subseteq \Omega(G)$. So $E(G) \cap \{45,14,15\}=\emptyset$. But in this case, $\{1,4,5\} \in \Omega(G)-F$.

The graph of $\Omega(G)$ is $G$ too. So $F_1=F$ and $F_2=\Omega(G)$ examplify Fact (2).
\end{example}

\begin{example}
$$F=\{\{1,2\},\{1,3\},\{2,3\},\{1,4\}\}.$$
$\alpha(F)=2$, $\bigcup F=\{1,2,3,4\}$ and $\bigcap F=\emptyset$. But $F$ is not an hke collection, because its subcollection, $\Gamma=\{\{1,2\},\{1,3\},\{2,3\}\}$, is not a KE collection. Let $G$ be the graph of $F$. $\alpha(G)=3$ (because $\{1,2,3\} \in \Omega(G)$) and $G$ is not a KE graph. 
\end{example}

%%%%%%%%%%%%%%%
%%
%%%%%%%%%%%%%%%
%\section{The Structure of the Collection of KE Graphs}

\begin{problem}
Let $G_1=(V_1,E_1)$ and $G_2=(V_2,E_2)$ be two KE graphs. We say that $G_1$ and $G_2$ are \emph{equivalent} when $\Omega(G_1) \cong \Omega(G_2)$. Find connections between equivalent graphs.
\end{problem}

\begin{definition}\label{the definition of the typical KE graph}
\emph{The typical KE graph for $\alpha$} is the following bipartite graph:
$V(G)=[2\alpha]$ and $E(G)=\{i,i+\alpha:i \in [\alpha]\}$.  
\end{definition}

\begin{theorem}\label{bipartite}
Let $G$ be a graph such that $V(G)=corona(G)$. Then $G$ is isomorphic to the typical KE graph for $\alpha(G)$ if and only if $\Omega(G)$ is a maximal hke collection.
\end{theorem}

\begin{proof}
Assume that $G$ is isomorphic to the typical KE graph for $\alpha$. Without loss of generality $V(G)=[2\alpha]$ and $E(G)=\{i,i+\alpha:i \in [\alpha]\}$, for some $\alpha$. So a maximal independent set in $G$ has $\alpha$ elements. Clearly, $\Omega(G)$ is the typical collection for $\alpha$. Hence, by Corollary \ref{the typical example is the unique one}, $\Omega(G)$ is a maximal hke collection.

Conversely, assume that $\Omega(G)$ is a maximal hke collection. By Theorem \ref{the typical example is the unique one}, without loss of generality, $\Omega(G)$ is the typical collection for $\alpha$. So $corona(G)=\bigcup \Omega(G)=[2\alpha]$. Therefore $V(G)=corona(G)=[2\alpha]$. 

By Definition \ref{the definition of the typical KE graph}, it remains to prove that $xy \in E(G)$ if and only if $|x-y|=\alpha$, or equivalently, $x$ is the $\Omega(G)$-dual of $y$.

\begin{claim}\label{xy in E implies duality}
Let $x$ and $y$ be two different vertices. If $xy \in E(G)$ then $x$ is the $\Omega(G)$-dual of $y$.
\end{claim}

\begin{proof}
For every $S \in \Omega(G)$, $S$ is independent and so the set $\{x,y\}$ is not included in $S$. Hence, by Corollary \ref{maximal independent is realy}, $x$ is the $\Omega(G)$-dual of $y$.
\end{proof}

\begin{claim}
Let $x$ and $y$ be two different vertices. Then $xy \in E(G)$ if and only if $x$ is the $\Omega(G)$-dual of $y$.
\end{claim} 

\begin{proof}
The first direction holds by Claim \ref{xy in E implies duality}. In order to prove the second direction, we assume that $xy \notin E(G)$. For the sake of a contradiction, assume that $x$ is the $\Omega(G)$-dual of $y$. So by the uniqueness of the $\Omega(G)$-dual of $y$, for every $z \in V(G)-\{x,y\}$, $z$ is not the $\Omega(G)$-dual of $y$. By Claim \ref{xy in E implies duality}, $zy \notin E(G)$. But by assumption, $xy \notin E(G)$. So for every $z \in V(G)-\{y\}$, $zy \notin E(G)$. 

Take $S \in \Omega(G)$ such that $x \in S$. Since $x \approx y$, by Corollary \ref{maximal independent is realy}, $y \notin S$. Hence, $S \cup \{y\}$ is an independent set of cardinality $\alpha+1$, a contradiction.
\end{proof}
Theorem \ref{bipartite} is proved.
\end{proof}

\section{The Subcollections of $\Omega(G)$ where $G$ is KE} 
We now can solve a problem of Jarden, Levit and Mandrescu.

The definition of a KE collection in \cite{dam}, is different from the definition in the current paper.
\begin{definition}\label{the old definition of a KE collection}
A relevant collection, $F$, is said to be \emph{a KE collection in the old sense} if $F \subseteq \Omega(G)$, for some graph $G$ and $|\bigcup F|+|\bigcap F|=2\alpha(G)$. 
\end{definition}

We restate \cite[Problem 3.1]{dam}:
\begin{problem}
Characterize the KE collections in the old sense.  
\end{problem}

\begin{theorem}\label{old KE iff hke}
Let $F$ be a collection of sets. The following conditions are equivalent:
\begin{enumerate}
 \item $F$ is an hke collection, \item $F$ is isomorphic to a subcollection of the typical collection for $\alpha(F)$, \item $F$ is included in $\Omega(G)$ for some KE graph $G$ and \item $F$ is a KE collection in the old sense. 
\end{enumerate}
\end{theorem}

\begin{proof}
We first prove that Clause (1) implies clause (2). 
Assume that $F$ is an hke collection. Let $F'$ be a maximal hke collection including $F$. Clearly, $\alpha(F')=\alpha(F)$. Let $G$ be the typical graph for $\alpha(F)$. By Theorem \ref{bipartite}, $\Omega(G)$ is a maximal hke. So by Theorem \ref{uniqueness of maximal}, $F'$ is isomorphic to $\Omega(G)$. Clause (2) is proved.

Easily, Clause (2) implies Clause (3).

If Clause (3) holds then $$|\bigcup F|+|\bigcap F| \leq |\bigcup \Omega(G)|+|\bigcap \Omega(G)|=2\alpha.$$
So Clause (4) holds.

Assume that Clause (4) holds, namely, $F$ is a KE collection in the old sense. So $F \subseteq \Omega(G)$ for some graph $G$ and $$|\bigcup F|+|\bigcap F|=2\alpha(G).$$ Let $\Gamma$ be a subcollection of $F$. By \cite[Corollary 2.7 and Corollary 2.9]{jlm}, $$|\bigcup \Gamma|+|\bigcap \Gamma|= 2\alpha(G).$$
\end{proof}

%%%%%%%%%%%%%%%%%%%%%%%%%%%%%%%%%%
%%%
%%%%%%%%%%%%%%%%%%%%%%%%%%%%%%%%%%

\section{Adding A Limitation on $|\bigcup F|$}

In this section, we prove that for every two positive integers $\alpha$ and $n$, the maximal cardinality of an hke collection, $F$, with $\alpha(F)=\alpha$ and $|\bigcup F|=n$ is $2^{n-\alpha}$. We conclude that for every positive integer $n$, the maximal cardinality of an hke collection, $F$, with $|\bigcup F|=n$ is $2^\alpha$ if $n=2\alpha$ or $n=2\alpha+1$.

Since every hke collection $F$ satisfies $\alpha(F) \leq |\bigcup F| \leq 2\alpha(F)$, the assumption $\alpha \leq n \leq 2\alpha$ in the following definition is needed.
\begin{definition}
Let $a(\alpha,n)$ be the maximal number of sets in an hke collection $F$ with $\alpha(F)=\alpha$ and $|\bigcup F|=n$, where $\alpha$ and $n$ are positive integers satisfying $\alpha \leq n \leq 2\alpha$.
\end{definition}

\begin{proposition}\label{a(alpha,n) leq 2^alpha}
For every two positive integers $\alpha$ and $n$ satsfying $\alpha \leq n \leq 2\alpha$, we have $a(\alpha,n) \leq 2^\alpha$. 
\end{proposition}

\begin{proof}
By Theorem \ref{maximal iff has 2^alpha elements}.
\end{proof}

\begin{proposition}\label{the new b(alpha)}
For every positive integer $\alpha$, we have $a(\alpha,2\alpha)=2^\alpha$.
\end{proposition}

\begin{proof}
By Proposition \ref{a(alpha,n) leq 2^alpha}, $a(\alpha,2\alpha) \leq 2^\alpha$.
The typical example for $\alpha$ exemplifies the converse. 
\end{proof}
 
In order to prove that $a(\alpha,n)=2^{n-\alpha}$, it remains to prove that $a(\alpha,n)=a(n-\alpha,2n-2\alpha)$. Propositions \ref{6.4} and \ref{6.5} are steps towards this goal.

\begin{proposition}\label{adding or subtracting}\label{6.4}
Let $\alpha,n$ and $d$ be positive integers with $\alpha \leq n \leq 2\alpha$.
\begin{enumerate}
\item $a(\alpha,n) \leq a(\alpha-d,n-d)$ whenever $a(\alpha-d,n-d)$ is defined (namely, $d \leq 2\alpha-n$) and \item $a(\alpha,n) \leq a(\alpha+d,n+d)$.
\end{enumerate}
\end{proposition} 
\begin{proof}
Let $F$ be an hke collection with $\alpha(F)=\alpha$ and $|\bigcup F|=n$. 

(1) $|\bigcap F|=2\alpha-n$. Let $C$ be a subset of $\bigcap F$ of cardinality $d$. Define $$F'=\{A-C:A \in F\}.$$  

Since $|\bigcup F'|=n-d$, $\alpha(F')=\alpha-d$ and $|F'|=|F|$, it remains to prove that $F'$ is an hke collection. 

Let $\Gamma'$ be a non-empty subcollection of $F'$. So $$\Gamma'=\{A-C:A \in \Gamma\},$$ for some non-empty subcollection  $\Gamma$ of $F$. We have to show that $\Gamma'$ is a KE collection, or equivalently, $$|\bigcup \Gamma'|+|\bigcap \Gamma'|=2(\alpha-d).$$ 

But
$$\bigcup \Gamma'=\bigcup \Gamma-C$$
and 
 $$\bigcap \Gamma'=\bigcap \Gamma-C.$$ So $$|\bigcup \Gamma'|+|\bigcap \Gamma'|=|\bigcup \Gamma|-d+|\bigcap \Gamma|-d=2\alpha-2d.$$

(2) Define $F'=\{A \cup C:A \in F\}$, where $C$ is a fixed set of cardinality $d$ with $C \cap \bigcup F=\emptyset$. Since $|\bigcup F'|=n+d$, $\alpha(F')=\alpha+d$ and $|F'|=|F|$, it remains to prove that $F'$ is an hke collection. 

Let $\Gamma'$ be a non-empty subcollection of $F'$. So $$\Gamma'=\{A \cup C:A \in \Gamma\},$$ for some non-empty subcollection  $\Gamma$ of $F$. We have to show that $\Gamma'$ is a KE collection, or equivalently, $$|\bigcup \Gamma'|+|\bigcap \Gamma'|=2(\alpha+d).$$ 

But
$$\bigcup \Gamma'=\bigcup \Gamma \cup C$$
and 
 $$\bigcap \Gamma'=\bigcap \Gamma \cup C.$$ So $$|\bigcup \Gamma'|+|\bigcap \Gamma'|=|\bigcup \Gamma|+d+|\bigcap \Gamma|+d=2\alpha+2d.$$

\end{proof}

\begin{proposition}\label{ignoring the intersection}\label{6.5}
Let $\alpha$ and $n$ be positive integers satisfying $\alpha \leq n \leq 2\alpha$.
If $-\infty<d \leq 2\alpha-n$, then $$a(\alpha,n)=a(\alpha-d,n-d).$$
\end{proposition}

\begin{proof}
By Proposition \ref{6.4}, it is enough to show that $a(\alpha-d,n-d)$ is defined. The assumption $d \leq 2\alpha-n$ is equivalent to $\alpha' \leq n' \leq 2\alpha'$ where $\alpha'=\alpha-d$ and $n'=n-d$ [$d \leq 2\alpha-n$ yields $n' \leq 2\alpha'$, for  $2\alpha'=2\alpha-2d=2\alpha-d-d \geq 2\alpha-d+(n-2\alpha)=n-d=n'$]. So $a(\alpha-d,n-d)$ is defined. 
\end{proof}

\begin{theorem}\label{corollary a(alpha,n)}
$$a(\alpha,n)=2^{n-\alpha}$$ holds for each two positive integers $\alpha$ and $n$ satisfying $\alpha \leq n \leq 2\alpha$.
\end{theorem}

\begin{proof}
By Proposition \ref{ignoring the intersection} (where $2\alpha-n$ stands for $d$), we have $$a(\alpha,n)=a(n-\alpha,2n-2\alpha).$$ But by Proposition \ref{the new b(alpha)}, $$a(n-\alpha,2n-2\alpha)=2^{n-\alpha}.$$
\end{proof} 

Let $\lceil n \rceil$ denote the integral value of $n$.
\begin{corollary}\label{computing b(alpha) and c(n)}
Let $n$ be a positive integer. The maximal cardinality of an hke collection, $F$, with $|\bigcup F|=n$ is $2^{\lceil \frac{n}{2} \rceil}$. 
\end{corollary}

\begin{proof}
Let $c(n)$ be the maximal cardinality of an hke collection, $F$, with $|\bigcup F|=n$. 

$$c(n)=max\{a(\alpha,n):\alpha \leq n \leq 2\alpha\}.$$

By Theorem \ref{corollary a(alpha,n)}, $a(\alpha,n)=2^{n-\alpha}$. So $a(\alpha,n)$ is maximal, when $\alpha$ is the minimal integer such that $n \leq 2\alpha$, or equivalently, $\alpha=\lceil \frac{n+1}{2} \rceil$. Therefore $c(n)=a(\lceil\frac{n+1}{2}\rceil,n)$. If $n=2k$ then $c(n)=a(\lceil\frac{n+1}{2}\rceil,n)=a(k,2k)=2^k$ and if $n=2k+1$ then $c(n)=a(\lceil\frac{n+1}{2}\rceil,n)=a(k+1,2k+1)=a(k,2k)=2^k$. In any case, $c(n)=2^{\lceil \frac{n}{2} \rceil}$.
\end{proof}

Connections between the current paper, \cite{critical}, \cite{criticallema} and \cite{criticalshort} should be studied.

\bibliographystyle{amsplain}
\bibliography{..//..//lit}

%\begin{thebibliography}{10}                                                                                         
%\bibitem[Le]{le} Vadim E. Levit and Eugen Mandrescu, A set and collection lemma, The Electronic Journal of Combinatorics, 21(1) (2014)

%\bibitem[JaRoLe]{jarole} Adi Jarden, Robert Shwartz and Vadim E. Levit, Matchings in graphs and groups, work in progress

%\bibitem [2]{2} A.Abdollahi, S. Akbary and H.R. Maimani, Non-commuting graph of a group, J. Algebra, 28: 468-492 (2006)

%\bibitem [3]{3} Z. Raza, S. Faizi, Non-commuting graph of a finitely presented group, Sci.Int.(Lahore),25(4),883-885 (2013)

%\bibitem [4]{4} A. Asghar Talebi, Non-commuiting graphs on dihedral group, Int.J.Algebra, 2(2):957-961 (2008)
%\bibitem[JLM]{jlm} Adi Jarden, Vadim E.Levit and Eugen Mandrescu, Monotonic Properties of Collections of Maximum
%Independent Sets of a Graph, a work in progress, preprint avaliable at: http://arxiv.org/pdf/1506.00249.pdf
%\end{thebibliography}

\end{document}